\documentclass[12pt]{amsart}
\usepackage{amsmath,amssymb,amsthm,amscd}
\newtheorem{formula}{}[section]
\newtheorem{proposition}[formula]{Proposition}
\newtheorem{corollary}[formula]{Corollary}
\newtheorem{lemma}[formula]{Lemma}
\newtheorem{theorem}[formula]{Theorem}
\theoremstyle{definition}
\newtheorem{definition}[formula]{Definition}
\newtheorem{example}[formula]{Example}
\theoremstyle{remark}
\newtheorem*{remark}{Remark}

\begin{document}

\title[Singular Virasoro vectors and cohomology]
{Singular Virasoro vectors  and Lie algebra cohomology} 
\author{Dmitri Millionschikov}
\address{Moscow State University, Department of Mathematics and Mechanics,
Leninskie gory 1, 119992 Moscow, Russia}
\email{million@higeom.math.msu.su}
\date{26 May, 2014}
\thanks{Supported by the grant of Russian Scientific Foundation  N 14-11-00414}
\begin{abstract}
We present an explicit formula for a new family of Virasoro singular vectors.
As a corrolary we get formulas for differentials of Feigin-Fuchs-Rocha-Carridi-Wallach  resolution of the the positive nilpotent part of Virasoro (or Witt) algebra $L_1$.
\end{abstract}

\maketitle
\section*{Introduction}
The main goal of the paper is to present explicit formulae for all differentials $\delta_k$ of
the Feigin-Fuchs-Rocha-Carridi-Wallach-resolution (\cite{FeFu}, \cite{RochWall}), a free resolution of a one-dimensional $L_1$-module ${\mathbb C}$.  By $L_1$ we denote the positive part of the Witt algebra. The differentials $\delta_k$ are $2\times2$-matrices
whose elements belong to the universal enveloping algebra $U(L_1)$. More precisely, the matrix $\delta_k$ can be expressed by means of singular vectors $S_{p,q}(t)v$ in Verma modules over the Virasoro Lie algebra (Virasoro singular vectors $S_{p,q}(t) \in U(L_1)$), $t$ is a complex parameter:
$$
\delta_1=\left( \begin{matrix} S_{1,1}({-}\frac{3}{2}) , \;  S_{1,2}({-}\frac{3}{2}) 
\end{matrix}\right), \; 
\delta_k=\left(
\begin{matrix}
S_{1,3k{+}1}({-}\frac{3}{2}) & S_{2k{+}1,2}({-}\frac{3}{2})\\
 {-}S_{2k{+}1,1}({-}\frac{3}{2}) & {-}S_{1,3k{+}2}({-}\frac{3}{2})
\end{matrix}\right), 
$$
It is easy to find the first few singular vectors, for instance 
$$
S_{1,1}(t)=e_1 , \;  S_{1,2}(t)=e_1^2{+}t^{{-}1}e_2, \;
S_{3,1}(t){=}e_1^3{+}4te_2e_1{+}(4t^2{+}2t)
e_3.
$$
However, further computational complexity starts to increase exponentially. 
Fuchs and Feigin proposed to look for the operators $S_{p,q}(t)$ as
\begin{equation}
S_{p,q}(t)=e_1^{pq}{+}
\sum_{\begin{array}{c}pq \ge i_1 \ge \dots \ge i_s\ge 1\\i_1+\dots+i_s=pq
\end{array}.}P_{p,q}^{i_1,\dots,i_s}(t)e_{i_1}\dots
e_{i_s},
\end{equation}
where $P_{p,q}^{i_1,\dots,i_s}(t)$ are Laurent polynomials in the complex variable $t$. Feigin, Fuchs and later Astashkevich had found some properties of these Laurent polynomials, however no explicit general formula have been found..

For a long time no one could get to write an explicit formula for
the $S_{p,q}(t)$. At the end of 80s Benoit
and Saint-Aubin \cite{BenSA} found a beautiful explicit expression  for one family of singular vectors $S_{p,1}(t)$. Hence three of the four matrix elements of $D_k$ we know (due to the property $S_{p,q}(t)=S_{q,p}(t^{-1})$), but the fourth element $S_{2k{+}1,2}\left({-}\frac{3}{2} \right)$ remained unknown. The key idea of ​​the approach by Benoit and Saint-Aubin was the idea to consider expansions of $S_{p,1}(t)$ in all monomials 
$$e_{i_1}e_{i_2}\dots e_{i_s}, i_1{+}i_2{+}\dots{+}i_s{=}pq,$$ 
which correspond to all unordered partitions of $pq$, not only to ordered partitions $i_1 \ge i_2 \ge \dots \ge i_s\ge 1$, as did Feigin and Fuchs. Of course this extends the formula for $S_{p,q}(t)$, a linear combination of this type is not unique, but in some cases (the calculation of the cohomology with different coefficients) it helps to get interesting combinatorial formulas \cite{Mill}.

Bauer, Di Francesco, Itzykson, Zuber \cite{BDFIZ} have found a very elegant proof of singularity of vectors $S_{p,1}(t)v$. Moreover, using the Benoit-Saint-Aubin formula as a starting point, Bauer, Di Francesco, Itzykson, Zuber \cite{BDFIZ} presented a complete and straightforward algorithm for finding all singular vectors. However, their algorithm encounters technical difficulties for $S_{p,q}(t), p,q \ge 3$
and it is not still clear whether it is possible to get an explicit formula in general case by means of it. The Benoit-Saint-Aubin formula, examples $S_{2,2}(t)$ and $S_{3,2}(t)$
considered in \cite{BDFIZ} has allowed us to guess the general explicit formula for 
$S_{2,p}(t)$. We prove the singularity of vectors $S_{2,p}(t)v$ following \cite{BDFIZ}  the  proof for another family $S_{p,1}(t)v$.

\section{Singular vectors of Verma modules over Virasoro Lie algebra}

\label{Singular_vectors} The Virasoro algebra  Vir is infinite
dimensional Lie algebra, defined by its basis  $\{z, e_i, i \in
{\mathbb Z}\}$ and commutator relations:
$$
[e_i,z]=0, \; \forall i \in {\mathbb Z}, \quad
[e_i,e_j]=(j-i)e_{i+j}+\frac{j^3-j}{12}\delta_{-i,j}z.
$$
Vir is one-dimensional central extension of the Witt algebra $W$
(the one-dimensional center is spaned by $z$).
\begin{remark}
In \cite{FeFu1, FeFu2, FeFuRe} the symbol $L_1$ denotes the
"positive part" of the Witt algebra  $W$ (or the Virasoro
algebra $Vir$), i.e. the algebra of {\it polynomial} vector fields
on the line ${\mathbb R}$ which vanish at the origin together with
their first derivatives. We will use further the symbol  $L_1$ for
the notation of both the algebras. This will not cause the
confusion because the cohomology $H^*(L_1)$ and $H^*(W_+)$ are
isomorphic \cite{G, Fu}.
\end{remark}

A $Vir$-module $V(h,c)$ is called a Verma module over the Virasoro
algebra if it is free as a module over the universal enveloping
algebra $U(L_1)$ of the subalgebra $L_1 \subset {\rm Vir}$ and it
is generated by some vector $v$ such that
$$
zv=cv, \; e_0v=hv,\quad e_iv=0, \; i <0,
$$
where $c,h \in {\mathbb C}$. As a vector space $V(h,c)$ can be
defined by its infinite basis
$$
v, \; e_{i_1}\dots e_{i_s}v, \quad i_1\ge i_2 \ge \dots \ge i_s,
\; s \ge 1.
$$
A Verma module  $V(h,c)$ is ${\mathbb Z}_+$-graded module:
$$
V(h,c)=\bigoplus_{n=0}^{+\infty} V_n(h,c), \quad V_n(h,c)=\langle
e_{i_1}\dots e_{i_s}v, \; i_1+\dots+i_s=n\rangle.
$$
$V_n(h,c)$ is an eigen-subspace  of the operator $e_0$ that
corresponds to the eigenvalue $(h+n)$:
$$
e_0(e_{i_1}\dots e_{i_s}v)=(h+i_1+\dots+i_s)e_{i_1}\dots e_{i_s}v.
$$
In addition to this, $zw=cw$ for all $w \in V(h,c)$.
\begin{definition}
A nontrivial vector $w\in V(h,c)$ is called singular if $e_iw=0$
for all $i <0$.
\end{definition}
\begin{remark}
A subalgebra $Vir^-$ which is spanned  by 
$\{e_i, i < 0\}$ is multiplicatively generated by two elements $e_{-1}$ and $e_{-2}$. Hence a vector
$w \in V(h,c)$ is singular if and only if 
$$
e_{-1}w=e_{-2}=0.
$$
\end{remark}

A homogeneous singular vector $w \in V_n(h,c)$ with the grading
equal to $n$ generates in $V(h,c)$ a submodule that is isomorphic
to $V(h+n,c)$.

It is not difficult to present first examples of singular vectors. At the first level
$n=1$ there is a singular vector $w_1 \in V_1(h,c)$ if and only if $h=0$.
Indeed the subspace $V_1(h,c)$ is one-dimensional and it is spanned by 
$e_1v$, but on the another hand 
$$
e_{-1}e_1v=2e_0v=2hv, \quad e_{-2}e_1v=0.
$$

A two-dimensional subspace coincides with the span of vectors $e_1^2v$ and $e_2v$. The
condition that the vector $w_2 \in V_2(h,c)$ is annihilated by the operator $e_{-1}$ is equivalent to the condition that (up to multiplication by a constant) the vector $w_2$ is equal to $e_1^2v-\frac{2}{3}(2h+1)e_2v$. On the another hand $e_{-2}$  annihilates  $w_2$ if and only if the parameters $h$ and $c$ of a Verma module $V(h,c)$ are related by
$$
\label{V_2_equation}
6h-\frac{2}{3}(2h+1)\left(4h+\frac{c}{2}\right)=0.
$$

It is easy to verify that the set of solutions of this equation can be parametrized in a following way
\begin{equation}
c(t)=13+6t+6t^{-1}, \quad h(t)=-\frac{3}{4}t-\frac{1}{2}, 
\end{equation}
where $t \ne 0$ runs the complex numbers (or runs the reals,  it depends on the task).

The last remark can be reformulated as follows: for any value of $t$  there is a unique (up to multiplication by a constant) singular vector $w_2=e_1^2v+te_2v, w_2 \in V_2(h(t), c(t))$, where
$h(t)$ and $c(t)$ are defined by equations (\ref{V_2_equation}).

\begin{theorem}[\cite{Kac}, \cite{FeFu, FeFu2}]
\label{Kac_theorem} In the Verma module $V(h,c)$ there is a
singular vector $w \in V_n(h,c)$ with the grading not higher than
$n$ if and only if, when two natural numbers $p$ and $q$ can been
found and also a complex number $t$ such that
\begin{equation}
\begin{split}
pq \le n, \quad c=c(t)=13+6t+6t^{-1},\\
h=h_{p,q}(t)=\frac{1-p^2}{4}t+\frac{1-pq}{2}+\frac{1-q^2}{4}t^{-1}.
\end{split}
\end{equation}
\end{theorem}
In particular, the following assertion holds \cite{Fu2}: with
the fixed natural numbers $p$ and $q$ and with an arbitrary
complex number $t$ the Verma module $V(h_{p,q}(t),c(t))$ contains
a singular vector $w_{p,q}(t)$ of degree $pq$, moreover the vector
$w_{p,q}(t)$ is determined unambiguously up to a multiplication by
some scalar:
$$
w_{p,q}(t)=S_{p,q}(t)v=\sum_{|I|=pq}P_{p,q}^I(t)e_Iv=
\sum_{i_1+\dots+i_s=pq}P^{i_1,\dots,i_s}_{p,q}(t)e_{i_1}\dots
e_{i_s}v,
$$
where $S_{p,q}(t)$ denotes some element of the universal
enveloping algebra $U(L_1)$. The coefficients $P_{p,q}^I(t)$
depend polynomially on $t$ and $t^{-1}$. We assume that the
coefficient $P^{1,\dots,1}_{p,q}(t)$ is equal to one. Obviously
that $S_{p,q}(t)=S_{q,p}(t^{-1})$.

\begin{theorem}[Benoit, Saint-Aubin]
\begin{equation}
\label{Benoit_SA}
S_{p,1}(t)=
\sum_{\begin{array}{c}i_1,\dots,i_s\\i_1+\dots+i_s=p
\end{array}.}c_p(i_1,\dots,i_s)t^{p-s}e_{i_1}\dots
e_{i_s}.
\end{equation}
where the sums are all over all partitions of $p$ by positive
numbers without any ordering restriction, and the coefficients
$c_p(i_1,\dots,i_s)$ are defined by the formulas
\begin{equation}
\label{coefficient_c} 
c_p(i_1,\dots,i_s)=
\frac{(p{-}1)!^2}{\prod_{l=1}^{s{-}1} \left((\sum_{q=1}^l i_q)(p-\sum_{q=1}^l i_q)\right)}
\end{equation}
\end{theorem}

\begin{example}
$$
S_{1,1}(t)=e_1,\;\;S_{2,1}(t)=e_1^2+te_2,\;\;S_{3,1}(t)=e_1^3+t(2e_1e_2+2e_2e_1)+4t^2e_3.
$$
\end{example}

\begin{theorem}
\label{main_theorem}
Let $V$ be a Verma module over the Virasoro algebra $Vir$. $V$ generated by the vector $v$ and such that $V$
corresponds to the (complex) parameter $t$:
with 
$$
c{=}13{+}6t{+}6t^{-1}, h{=}{-}\frac{(p{-}1{+}t)(t^{-1}\left(p{+}1){+}3\right)}{4}.
$$

Let consider an element of universal enveloping algebra $U(L_1)$ definend by the formula
\begin{equation}
\label{new_family}
S_{2,p}(t)=
{\sum_{\begin{array}{c}i_1,\dots,i_s\\i_1{+}{\dots}{+}i_s{=}2p
\end{array}.}}
f_p(i_1,\dots,i_s)e_{i_1}\dots
e_{i_s}.
\end{equation}
where the sums are all over all partitions of $2p$ by positive
numbers without any ordering restriction, and the coefficients (that are rational functions on $t$)
$f_p(i_1,\dots,i_s)$ are defined by the formulas
\begin{equation}
\label{coefficient_f} 
f_p(i_1{,}{\dots}{,}i_s)
{=}\frac{(2p{-}1)!^2 (2t)^{s{-}2p} \prod\limits_{r=1}^{2p{-}1}(p{-}t{-}r) \prod\limits_{m=1}^s \left( i_m(2t{+}1){+}2(p{-}t{-}\sum\limits_{n=1}^m i_n)\right)}
{\prod\limits_{l=0}^{2p{-}1} (2p{-}1{-}2l)\prod\limits_{l=1}^{s-1}\left((\sum\limits_{n=1}^l i_n)(2p{-}\sum\limits_{n=1}^l i_n)(p{-}t{-}\sum\limits_{n=1}^l{i_n})\right)}.
\end{equation}
Then $S_{2,p}(t)v$ is a singular vector of $V$:
\begin{equation}
\label{singular_cond}
e_{-k}S_{2,p}v=0, \quad k \in {\mathbb N}.
\end{equation}
\end{theorem}

For instance $S_{2,1}=e_1^2{+}te_2$ and
\begin{equation}
\begin{split}
S_{2,2}(t)=e_1^4{+}4te_1e_2e_1{+}
\frac{(1{-}t^2)}{t}(e_1^2e_2{+}e_2e_1^2){+}\frac{(1{-}t^2)^2}{t^2}e_2^2{+}\\{+}\frac{(1{+}t)(4t{-}1)}{t}e_1e_3{+}\frac{(1{-}t)(4t{+}1)}{t}e_3e_1{+}\frac{3(1{-}t^2)}{t}e_4.
\end{split}
\end{equation}

\begin{proof}
We define the sequence of vectors $v^{(0)}, v^{(1)},\dots, v^{(2p{-}1)}$ by 
$$
v^{(0)}=v
$$ 
and the following recursive relation
 for $k=1,\dots, 2p{-}1$:
\begin{equation}
v^{(k)}=\frac{2t\sum\limits_{j{=}1}^k\left( (j{-}1)(2t{-}1){+}2k{-}2p{-}1\right)e_jv^{(k{-}j)}
}{k(2p{-}k)(k{-}p{-}t)}
\end{equation}
For instance
\begin{equation}
\begin{split}
v^{(1)}=\frac{2t(1{-}2p)e_1v^{(0)}}{(2p{-}1)(1{-}p{-}t)}, \\ 
v^{(2)}=\frac{2t\left(2(t{+}1{-}p)e_2v^{(0)} {+} (3{-}2p)e_1v^{(1)}\right)}{2(2p{-}2)(2{-}p{-}t)}.
\end{split}
\end{equation}
After that we define the vector $w \in V_{2p}$ by the formula:
\begin{equation}
w=2t\sum\limits_{j{=}1}^{2p}\left( (j{-}1)(2t{-}1){+}2p{-}1\right)e_jv^{(2p{-}j)}.
\end{equation}

\begin{lemma}
\begin{equation}
e_{-1}v^{(k)}=-(p{+}2{-}k{+}3t)v^{(k{-}1)}, \quad k=1,\dots, 2p{-}1.
\end{equation}
and the vector $w$ annihilates the operator $e_{-1}$:
$$
e_{-1}w=e_{-1}\left( 2t\sum\limits_{j{=}1}^{2p}\left( (j{-}1)(2t{-}1){+}2p{-}1\right)e_jv^{(2p{-}j)}\right)=0.
$$
\end{lemma}

\begin{proof}
We will prove the following formula for $k=1,\dots,2p$:
\begin{equation}
\begin{split}
e_{-1}\left(2t\sum\limits_{j{=}1}^k\left( (j{-}1)(2t{-}1){+}2k{-}2p{-}1\right)e_jv^{(k{-}j)}\right)=\\
=-(p{+}2{-}k{+}3t)k(2p{-}k)(k{-}p{-}t)v^{(k{-}1)}.
\end{split}
\end{equation}
We proceed by recursion on $k$. The recursion base is
$$
e_{-1}\left(2t(2{-}2p{-}1)e_1v^{(0)} \right)=2t(1{-}2p)2hv^{(0)}=-(p{+}1{+}3t)(2p{-}1)(1{-}p{-}t)v^{(0)}.
$$
The recursive step is the following calculation
\begin{equation}
\begin{split}
e_{-1}\left(2t\sum\limits_{j{=}1}^{k}\left( (j{-}1)(2t{-}1){+}2k{-}2p{-}1\right)e_jv^{(k{-}j)}\right)=\\
=2t\sum\limits_{j{=}1}^{k}\left( (j{-}1)(2t{-}1){+}2k{-}2p{-}1\right)\left((j{+}1)e_{j{-}1}v^{(k{-}j)} 
{+}e_{j}e_{-1}v^{(k{-}j)}\right)=\\
=4t(2k{-}2p{-}1)e_0v^{(k{-}1)}{+}2t\sum\limits_{j{=}2}^{k}\left( (j{-}1)(2t{-}1){+}2k{-}2p{-}1\right)(j{+}1)e_{j{-}1}v^{(k{-}j)}{+}\\ 
{-}2t\sum\limits_{j{=}1}^{k{-}1}\left( (j{-}1)(2t{-}1){+}2k{-}2p{-}1\right)(p{+}2{-}k{+}j{+}3t)e_{j}v^{(k{-}j{-}1)}=\\
=t(2k{-}2p{-}1)4(h{+}k{-}1)v^{(k{-}1)}{+}\\
{-}(p{-}k{+}1{+}3t)2t\sum\limits_{j{=}1}^{k{-}1}\left((j{-}1)(2t{-}1){+}2k{-}2p{-}3\right)e_{j}v^{(k{-}j{-}1)}.
\end{split}
\end{equation}
In this chain of equalities we replaced $e_{-1}e_j$ by $(j{+}1)e_{j{-}1}{+}e_je_{-1}$ and used the formula
for $e_{-1}v^{(k{-}j)}$ which we considered true for $j=1,\dots, k$ by the induction hypothesis. 

Then we shifted $j{=}j'{+}1$ the summation index in the sum
$$
\sum\limits_{j{=}2}^{k}\left( (j{-}1)(2t{-}1){+}2k{-}2p{-}1\right)(j{+}1)e_{j{-}1}v^{(k-j)}
$$
and used the following equality
\begin{equation}
\begin{split}
 (j(2t{-}1){+}2k{-}2p{-}1)(j{+}2)-\\
 -((j{-}1)(2t{-}1){+}2k{-}2p{-}1)(p{+}2{-}k{+}j{+}3t)=\\
={-}(p{-}k{+}1{+}3t)\left((j{-}1)(2t{-}1){+}2k{-}2p{-}3\right).
\end{split}
\end{equation}

It follows from the definition of $v^{(k{-}1)}$ that
\begin{equation}
\begin{split}
2t\sum\limits_{j{=}1}^{k{-}1}\left((j{-}1)(2t{-}1){+}2k{-}2p{-}3\right)e_{j}v^{(k{-}j{-}1)}=\\
=(k{-}1)(2p{-}k{+}1)(k{-}1{-}p{-}t)v^{(k{-}1)}.
\end{split}
\end{equation}
We remark that 
$$
4t(h{+}k{-}1)=\left(-(p{-}1{+}t)(p{+}1{+}3t){+}4kt{-}4t)\right)
$$
and finish our calculations.
\begin{equation}
\label{resulting_equality}
\begin{split}
e_{-1}\left(2t\sum\limits_{j{=}1}^{k}\left( (j{-}1)(2t{-}1){+}2k{-}2p{-}1\right)e_jv^{(k{-}j)}\right)=\\
=(2k{-}2p{-}1)\left({-}(p{-}1{+}t)(n{+}1{+}3t){+}4kt{-}4t\right)v^{(k{-}1)}{-}\\
{-}(p{-}k{+}1{+}3t)(k{-}1)(2p{-}k{+}1)(k{-}1{-}p{-}t)v^{(k{-}1)}=\\
=-k(2p{-}k)(k{-}p{-}t)(p{+}2{-}k{+}3t)v^{(k{-}1)}.
\end{split}
\end{equation}
If we divide the resulting equality (\ref{resulting_equality})  by $k(2p{-}k)(k{-}p{-}t)$, we obtain the required formula 
$$
e_{-1}v^{(k)}=-(p{+}2{-}k{+}3t)v^{(k-1)}.
$$ 

Taking $k=2p$ we will get 
$$
e_{-1}w=0.
$$
\end{proof}

\begin{lemma}
\begin{equation}
e_{-2}v^{(1)}=0, \quad e_{-2}v^{(k)}=-(p{+}4{-}k{+}5t)v^{(k{-}2)}, \quad k=2,\dots, 2p{-}1.
\end{equation}
and the vector $w$ annihilates the operator $e_{-2}$:
$$
e_{-2}w=e_{-2}\left( 2t\sum\limits_{j{=}1}^{2p}\left( (j{-}1)(2t{-}1){-}1\right)e_jv^{(2p{-}j)}\right)=0.
$$
\end{lemma}

\begin{proof}
We recursively prove this formula for $k=2,\dots,2p$:
\begin{equation}
\begin{split}
e_{-2}\left(2t\sum\limits_{j{=}1}^k\left( (j{-}1)(2t{-}1){+}2k{-}2p{-}1\right)e_jv^{(k{-}j)}\right)=\\
=-(p{+}4{-}k{+}5t)k(2p{-}k)(k{-}p{-}t)v^{(k{-}2)}.
\end{split}
\end{equation}
The starting point $k{=}1$ is evident
$$
e_{-2}\left(2t(2{-}2p{-}1)e_1v^{(0)} \right)=2t(1{-}2p) (3e_{-1}v^{(0)}+e_1e_{-2}v^{(0)})=0.
$$
Now we take $k{=}2$.
\begin{equation}
\begin{split}
e_{-2}\left(2t(2t{+}2k{-}2p{-}2)e_2v^{(0)} {+} 2t(3{-}2p)e_1v^{(1)}\right)=\\
=2t(3{-}2p)3e_{-1}v^{(1)}{+}2t(2t+2-2p)(4e_0{+}\frac{1}{2}z)v^{(0)}=\\
=2t\left(-3(3{-}2p)(p{+}1{+}3t){+}2(t{+}1{-}p)(4h{+}\frac{13}{2}{+}3t{+}3t^{-1})\right)v^{(0)}=\\
=-(p{+}2{+}5t)2(2p{-}2)(2{-}p{-}t)v^{(0)}.
\end{split}
\end{equation}

Now let consider the recursive step.
\begin{equation}
\begin{split}
e_{-2}\left(2t\sum\limits_{j{=}1}^{k}\left( (j{-}1)(2t{-}1){+}2k{-}2p{-}1\right)e_jv^{(k{-}j)}\right)=\\
=2t\sum\limits_{j{=}1}^{k}\left( (j{-}1)(2t{-}1){+}2k{-}2p{-}1\right)\left((j{+}2)e_{j{-}2}v^{(k{-}j)} 
{+}e_{j}e_{-2}v^{(k{-}j)}\right)=\\
{=}6t(2k{-}2p{-}1)e_{-1}v^{(k{-}1)}{+}4t(t{+}1{-}p)(4(h{+}k{-}2){+}\frac{13}{2}{+}3t{+}3t^{-1})v^{(k{-}2)}{+}\\
+2t\sum\limits_{j{=}3}^{k}\left( (j{-}1)(2t{-}1){+}2k{-}2p{-}1\right)(j{+}2)e_{j{-}2}v^{(k{-}j)}-\\ 
{-}2t\sum\limits_{j{=}1}^{k{-}2}\left( (j{-}1)(2t{-}1){+}2k{-}2p{-}1\right)(p{+}4{-}k{+}j{+}5t)e_{j}v^{(k{-}j{-}2)}=
\end{split}
\end{equation}
We used the induction assumption 
for $e_{-2}v^{(k{-}j)}, j=1,\dots, k{-}2$.

Now we shift the summation index $j'=j{-}2$ in the sum
$$
\sum\limits_{j{=}3}^{k}\left( (j{-}1)(2t{-}1){+}2k{-}2p{-}1\right)(j{+}2)e_{j{-}1}v^{(k-j)},
$$
replace $3e_{-1}v^{(k{-}1)}$ by $-(p{+}3{-}k{+}3t)v^{(k{-}2)}$
and we have 

\begin{equation}
\begin{split}
e_{-2}\left(2t\sum\limits_{j{=}1}^{k}\left( (j{-}1)(2t{-}1){+}2k{-}2p{-}1\right)e_jv^{(k{-}j)}\right)=\\
=2tQv^{(k{-}2)}+
2t\sum\limits_{j{=}1}^{k{-}2}R_je_jv^{(k{-}j{-}2)}
\end{split}
\end{equation}
where
$$
Q{=}{-}3(2k{-}2p{-}1)(p{+}3{-}k{+}3t){+}2(t{+}k{-}p{-}1)\left( 4h{+}4k{-}8{+}\frac{13}{2}{+}3t{+}3t^{-1}\right)
$$
and now we compute the coefficient $R_j$
\begin{equation}
\begin{split}
R_j{=}{-}\left( (j{-}1)(2t{-}1){+}2k{-}2p{-}1\right)(p{+}4{-}k{+}j{+}5t){+}\\
{+}(4{+}j)\left( (j{-}1)(2t{-}1){+}2k{-}2p{-}1\right)=\\
=-(p{-}k{-}2{+}5t)\left((j{-}1)(2t{-}1){+}2k{-}2p{-}5\right)
\end{split}
\end{equation}
The first factor of $R_j$ does not depend on $j$ and moreover
\begin{equation}
\begin{split}
2t\sum\limits_{j{=}1}^{k{-}2}R_je_jv^{(k{-}j{-}2)}=-(p{-}k{-}2{+}5t)(k{-}2)(2p{-}k{+}2)(k{-}p{-}t)v^{(k{-}2)}
\end{split}
\end{equation}
Hence
\begin{equation}
\begin{split}
\label{resultformul}
e_{-2}\left(2t\sum\limits_{j{=}1}^{k}\left( (j{-}1)(2t{-}1){+}2k{-}2p{-}1\right)e_jv^{(k{-}j)}\right)=\\
=\left(2tQ{-}(p{-}k{-}2{+}5t)(k{-}2)(2p{-}k{+}2)(k{-}p{-}t)\right)v^{(k{-}2)}{=}\\
=-k(2p{-}k)(k{-}p{-}t)(p{+}4{-}k{+}5t)v^{(k{-}2)}.
\end{split}
\end{equation}

We divide the resulting equality by $k(2p{-}k)(k{-}p{-}t)$ and obtain the required formula.

Taking $k=2p$ in (\ref{resultformul}) we have 
$$
e_{-2}w=0.
$$
This completes the proof of the lemma.
\end{proof}
Proof of the theorem follows from two lemmas. The only one thing we still have to do is to present
explicit formula for $w$.
\begin{equation}
\begin{split}
w=2t\sum\limits_{j_1{=}1}^{2p}\left( (j_1{-}1)(2t{-}1){+}2p{-}1\right)e_{j_1}v^{(2p{-}j_1)},\\
v^{(2p{-}j_1)}{=}\sum\limits_{j_2{=}1}^{2p{-}j_1}\frac{2t \left( (j_2{-}1)(2t{-}1){+}2p{-}2j_1{-}1\right)}{(2p{-}j_1)j_1(p{-}j_1{-}t)}{e_{j_2}}{v^{(2p{-}j_1{-}j_2)}},\\
v^{(2p{-}j_1{-}j_2)}{=}\sum\limits_{j_3{=}1}^{2p{-}j_1{-}j_2}\frac{2t \left( (j_3{-}1)(2t{-}1){+}2p{-}2j_1{-}2j_2{-}1\right)}{(2p{-}j_1{-}j_2)(j_1{+}j_2)(p{-}j_1{-}j_2{-}t)}{e_{j_3}}{v^{(2p{-}j_1{-}j_2{-}j_3)}}.
\end{split}
\end{equation}
Proceeding further step by step we obtain the formula
\begin{equation}
\begin{split}
w={\sum_{\begin{array}{c}j_1,\dots,j_s\\j_1{+}{\dots}{+}j_s{=}2p
\end{array}.}}
\frac{(2t)^s\prod\limits_{r=1}^{s}\left((j_r{-}1)(2t{-}1){+}2p{-}1{-}2\sum\limits_{q=1}^{r{-}1}j_q) \right)}{\prod\limits_{m=1}^{s-1}\left((\sum\limits_{q=1}^{m}j_q)(2p{-}\sum\limits_{q=1}^{m}j_q)(p{-}t{-}\sum\limits_{q=1}^{m}j_q)\right)}e_{j_1}\dots
e_{j_s}v.
\end{split}
\end{equation}
We need to calculate the coefficient  facing $e_1^{2p}$ in the expansion of $w$:
$$
w
=\frac{(2t)^{2p}\prod\limits_{k=0}^{2p{-}1}(2p{-}1{-}2k)}{(2p{-}1)!^2\prod\limits_{q{=}1}^{2p{-}1}(p{-}t{-}q)}
e_1^{2p}v+\dots.
$$
Finally we get
$$
S_{2,p}(t)v=\frac{(2p{-}1)!^2\prod\limits_{q{=}1}^{2p{-}1}(p{-}t{-}q)}{(2t)^{2p}\prod\limits_{k=0}^{2p{-}1}(2p{-}1{-}2k)}w.
$$
\end{proof}
\begin{example}
\begin{equation}
\begin{split}
S_{2,3}(t)=e_1^6{+}\frac{(4{-}t^2)}{3t}(e_1^4e_2{+}e_2e_1^4){+}
\frac{8(1{-}t^2)}{3t}(e_1^3e_2e_1{+}e_1e_2e_1^3){+}9te_1^2e_2e_1^2{+}\\ 
{+}3(4{-}t^2)(e_1^2e_2^2{+}e_2^2e_1^2){+}\frac{64(1{-}t^2)^2}{9t^2}e_1e_2^2e_1{+}\frac{(4{-}t^2)^2}{9t^2}e_2e_1^2e_2{+}\\
{+}\frac{8(1{-}t^2)(4{-}t^2)}{9t^2}(e_1e_2e_1e_2{+}e_2e_1e_2e_1){+}\frac{(4{-}t^2)^2}{t}e_2^3{+}\\
{+}\frac{4(4{-}t^2)(1{-}t^2)(9{-}16t^2)}{9t^4}e_3^2{+}\frac{6(1{+}t)(4t{-}1)}{t}e_1^2e_3e_1{+}
\frac{6(1{-}t)(4t{+}1)}{t}e_1e_3e_1^2{+}\\
{+}\frac{2(1{+}t)(2{+}t)(3{-}4t)}{3t^2}e_1^3e_3{+}\frac{2(1{-}t)(2{-}t)(3{+}4t)}{3t^2}e_3e_1^3{+}\\
{+}\frac{16(1{-}t^2)(1{+}t)(2{+}t)(3{-}4t)}{9t^3}e_1e_2e_3{+}\frac{16(1{-}t^2)(1{-}t)(2{-}t)(3{+}4t)}{9t^3}e_3e_2e_1{+}\\
{+}\frac{2(4{-}t^2)(2{+}t)(1{+}t)(3{-}4t)}{9t^3}e_2e_1e_3{+}\frac{2(4{-}t^2)(2{-}t)(1{-}t)(3{+}4t)}{9t^3}e_3e_1e_2{+}\\
{+}\frac{2(4{-}t^2)(1{-}t)(4t{+}1)}{t^2}e_1e_3e_2{+}\frac{2(4{-}t^2)(1{+}t)(4t{-}1)}{t^2}e_2e_3e_1{+}\\
{+}\frac{6(1{+}t)(2{+}t)(3t{-}1)}{t^2}e_1^2e_4{+}\frac{48(1{-}t^2)}{t}e_1e_4e_1{+}
\frac{6(1{-}t)(2{-}t)(3t{+}1)}{t^2}e_4e_1^2{+}\\
{+}\frac{(4{-}t^2)(2{+}t)(1{+}t)(3t{-}1)}{t^3}e_2e_4{+}\frac{(4{-}t^2)(2{-}t)(1{-}t)(3t{+}1)}{t^3}e_4e_2{+}\\
{+}\frac{4(1{-}t^2)(2{+}t)(8t{-}1)}{t^3}e_1e_5{+}\frac{4(1{-}t^2)(2{-}t)(8t{+}1)}{t^3}e_5e_1{+}\frac{20(1{-}t^2)(4{-}t^2)}{t^3}e_6.
\end{split}
\end{equation}
\end{example}

Obviously, coefficients $f_p(i_1,\dots, i_s)$ facing monomials $e_{i_1}\dots e_{i_s}$ (that correspond to unordered partitions of $2p$) in the expansion of $S_{2,p}(t)$ are not uniquely defined, because these monomials are linearily dependent  in $V$. There are, however, two coefficients $f_p(1,\dots,1)$ and $f_p(2,\dots, 2)$ facing $e_1^{2p}$ and $e_2^p$ respectively which, as it is easily seen, are uniquely determined. Let us calculate $f_p(2,\dots,2)$.
$$
f_p(2,\dots, 2)=\frac{\prod\limits_{q=1}^{p}\left(t^2{-}(p{+}1{-}2q)^2\right)}{t^p}.
$$
In particular, in our two previous examples we have
$$
f_2(2,2)=\frac{(t^2{-}1)^2}{t^2}, \quad f_3(2,2,2)=\frac{(t^2{-}4)^2t^2}{t^3}.
$$
It was proved in \cite{AstFu} that
$$
S_{k,l}(t)=(k{-}1)!^{2l}e_k^lt^{(k{-}1)l}+\dots+(l{-}1)!^{2k} e_l^kt^{-(l{-}1)k},
$$
where "$\dots$" denotes intermediate degrees in $t$.
we see that our coefficient $f_p(2,\dots,2)$ has prescribed asymptotic behavior with respect to $t$.
Hence 
$$
S_{2,p}(t)=e_2^pt^{p}+\dots
$$
which is consistent with our calculations.

\section{Singular vectors and cohomology}
Now we are going to consider Verma modules over the Virasoro algebra 
with $c=0$ that one can consider as Verma modules over the Witt algebra.

\begin{proposition}[Kac \cite{Kac}, Feigin and Fuchs \cite{FeFu, FeFu2}]
There is a singular vector $w_n$ in the homogeneous subspace
$V_n(0,0)$ of the Verma module $V(0,0)$ then and only then when
$n$ is equal to some pentagonal number $n=e_{\pm}(k)=\frac{3k^2\pm
k}{2}$.
\end{proposition}
It follows from the theorem  \ref{Kac_theorem} that if a Verma
module $V(h,0)$ (with $c{=}0$) has a singular vector $w_{p,q}(t)$ 
it  implies that $t{=}{-}\frac{3}{2}$ or $t{=}{-}\frac{2}{3}$. We will
fix the value $t=-\frac{3}{2}$ and we will write $S_{p,q}$ instead
of $S_{p,q}\left({-}\frac{3}{2}\right)$ for convenience in the
notations. Let us denote by $V\left(\frac{3k^2\pm k}{2}\right)$ a
submodule in the Verma module $V(0,0)$ generated by a singular
vector with the grading $\frac{3k^2\pm k}{2}$. The submodule
$V\left(\frac{3k^2\pm k}{2}\right)$ is isomorphic to the Verma
module $V\left(\frac{3k^2\pm k}{2},0\right)$.
\begin{proposition}[\cite{RochWall}, \cite{FeFu}]
The system of submodules  $V\left(\frac{3k^2\pm k}{2}\right)$ has
the following important properties:

1) the sum $V(1)+V(2)$ is the subspace of codimension one in
 $V(0)$;

2) $V\left(\frac{3k^2- k}{2}\right)\cap V\left(\frac{3k^2+
k}{2}\right)=V\left(\frac{3(k{+}1)^2{-} (k{+}1)}{2}\right)+
V\left(\frac{3(k{+}1)^2{+} (k{+}1)}{2}\right),k {\ge} 1$.
\end{proposition}
One can directly verify that the vectors  $e_1v$ and
$\left(e_1^2-\frac{2}{3}e_2\right)v$ are singular in the module
 $V(0,0)$ with the gradings  $1$ and $2$ respectively. Let consider
a submodule $V(1)$ generated by $w_1{=}e_1v$. It is isomorphic to the Verma module
$V(1,0)$ and it containes a singular vector 
$S_{1,4}w_1$. A Verma module $V(2,0) = V(2)$ (generated by the vector $w_2{=}S_{1,2}v$)
containes a singular vector 
$S_{3,1}w_2.$ 
Vectors $S_{1,4}w_1$ and $S_{3,1}w_2$  are both singular and they at the level $n{=}5$ in the Verma module $V(0,0)$. Hence they coincide
$$
w_5=S_{1,4}w_1=S_{3,1}w_2.
$$
Similarly, one can check the other equality
$$
w_7=S_{3,2}w_1=S_{1,5}.
$$
We see that singular vectors $w_5, w_7 \in V(1)\cap V(2)$ as well as the sum
$V(5) + V(7)$ of submodules generated by $w_5$ and $w_7$ respectively:
$$
V(5) + V(7) \subset  V(1)\cap V(2).
$$
The intersection $V(5) \cap V(7)$ containes two singular vectors
$$
w_{12}=S_{1,7}w_5=S_{5,1}w_7, \quad
w_{15}=S_{5,2}w_5=S_{1,8}w_7.
$$
The inclusions of  submodules $V\left(\frac{3k^2\pm k}{2}\right)$
provides us with an exact sequence  \cite{RochWall, FeFu, FeFu2}:
\begin{equation}
\label{resolution}
\begin{split}
\begin{CD}
{\dots}{\rightarrow} {V(\frac{3(k{+}1)^2{-}
(k{+}1)}{2}){\oplus}V(\frac{3(k{+}1)^2{+} (k{+}1)}{2})}
@>{\delta_{k{+}1}}>> V(\frac{3k^2{-}
k}{2}){\oplus}V(\frac{3k^2{+} k}{2}){\rightarrow}{\dots}\end{CD}\\
\begin{CD}
{\dots} @>{\delta_3}>> V(5){\oplus}V(7) @>{\delta_2}>>
V(1){\oplus}V(2) @>{\delta_1}>> V(0) @>{\varepsilon}>>
{\mathbb C} {\to} 0
\end{CD},
\end{split}
\end{equation}
where $\delta_k$ are defined with the aid of operators $S_{p,q}
\in U(L_1)$:
\begin{equation}
\begin{split}
\delta_{k+1}\left( \begin{array}{c}x \\y
\end{array}\right)=\left(\begin{array}{cc}S_{1,3k{+}1} & S_{2k{+}1,2}\\
{-}S_{2k+1,1}& {-}S_{1,3k{+}2} \end{array}\right)\left(
\begin{array}{c}x \\y \end{array}
\right), \; k \ge 1;\\
\delta_1\left(
\begin{array}{c}x \\y \end{array}\right)=\left(S_{1,1}, S_{1,2}\right)\left(
\begin{array}{c}x \\y \end{array}\right),
\end{split}
\end{equation}
and $\varepsilon$ is a projection to the one-dimensional ${\mathbb
C}$-submodule generated by the vector $v$.
\begin{theorem}[\cite{RochWall}, \cite{FeFu}]
The exact sequence (\ref{resolution}) considered as a sequence of
$L_1$-modules is a free resolution of the one-dimensional trivial
$L_1$-module ${\mathbb C}$.
\end{theorem}
\begin{corollary}[\cite{FeFu}]
Let $V$ be a $L_1$-module. Then the cohomology $H^*(L_1,V)$ is
isomorphic the cohomology of the following complex:
\begin{equation}
\label{resolution2}
\begin{CD}
{\dots}@<{d_{k{+}1}}<< V{\oplus}V @<{d_k}<< V{\oplus}V
@<{d_{k{-}1}}<< {\dots} @<{d_1}<< V{\oplus}V @<{d_0}<< V
\end{CD},
\end{equation}
with the differentials
\begin{equation}
\label{differential}
\begin{split}
d_k\left(\begin{array}{c}m_1 \\m_2 \end{array}\right)=\left(\begin{array}{cc}S_{1,3k{+}1} & {-}S_{2k+1,1}\\
S_{2k{+}1,2}& {-}S_{1,3k{+}2} \end{array}\right)\left(\begin{array}{c}m_1 \\m_2 \end{array}\right), \; k \ge 1;\\
d_0(m)=\left(\begin{array}{c}S_{1,1}m \\S_{1,2}m
\end{array}\right),\;\; m, m_1, m_2 \in V.
\end{split}
\end{equation}
\end{corollary}
Let consider the trivail one-dimensional module $V={\mathbb C}$. All operators
$d_k$ are trivial and we obtain the famous Goncharova theorem.
\begin{theorem}[\cite{G}] The space of $q$-cohomology
$H^q(L_1,{\mathbb C})$ is two-dimensional for all $q \ge 1$,
moreover it is the direct sum of its one-dimensional subspaces:
$$ H^q(L_1,{\mathbb C})=H_{\frac{3q^2 {-} q}{2}}^q(L_1,{\mathbb C})
\oplus H_{\frac{3q^2{+}q}{2}}^q(L_1,{\mathbb C}).  $$
\end{theorem}
The numbers $e_{\pm}(q)=\frac{3q^2 \pm q}{2}$ are
called Euler pentagonal numbers. 

\begin{remark}
The original proof of the properties of the free resolution of one-dimensional $L_1$-module
${\mathbb C}$ by Rocha-Carridi and Wallach \cite{RochWall} seriously used the Goncharova
theorem.
\end{remark}

Fuchs and Feigin studied $L_1$-cohomology with coefficients in graded modules $V=\oplus_i V_i$ only with
one-dimensional homogeneous components  $V_i$. We will define them
with the aid of the special basis $f_i, V_i=\langle f_{i}\rangle$
($j \in {\mathbb Z}$ in the infinite dimensional case and $j \in
{\mathbb Z}, m\le j\le n$ in the finite dimensional). For a given
graded $L_1$-module $V=\oplus_i V_i$ let us introduce the numbers
$\sigma_{p,q}(j) \in {\mathbb K}$ such that
$$
S_{p,q}f_j=\sigma_{p,q}(j)f_{j{+}pq}.
$$
\begin{example}
\label{F_lambda_mu} The well-known $L_1$-module $F_{\lambda,\mu}$
of tensor densities  \cite{Fu}:
$$ e_if_j=\left(j+\mu-\lambda(i+1) \right) f_{i+j}, \forall i \in
{\mathbb N}, j \in {\mathbb Z},
 $$
where $\lambda, \mu \in {\mathbb K}$ are two complex parameters.
\end{example}

\begin{corollary}[\cite{FeFu}]
Let $V=\oplus_i V_i$ be a graded $L_1$-module over the field
${\mathbb K}$. Then the one-dimensional cohomology $H^*_s(L_1,V)$
is isomorphic to the cohomology of the following complex:
\begin{equation}
\label{resolution3}
\begin{CD}
{\dots}@<{D_{k{+}1}}<< {\mathbb K}{\oplus}{\mathbb K} @<{D_k}<<
{\mathbb K}{\oplus}{\mathbb K}@<{D_{k{-}1}}<< {\dots}  @<{D_1}<<
{\mathbb K}{\oplus}{\mathbb K} @<{D_0}<< {\mathbb K}
\end{CD},
\end{equation}
where the differentials $D_k$ are assigned by the numerical
matrices
\begin{equation}
\label{D_k}
D_k{=}\left(\begin{array}{cc}\sigma_{1,3k{+}1}\left(s{+}\frac{3k^2{-}k}{2}\right)
& {-}\sigma_{2k+1,1}\left(s{+}\frac{3k^2{-}k}{2}\right)\\
\sigma_{2k{+}1,2}\left(s{+}\frac{3k^2{+}k}{2}\right)&
{-}\sigma_{1,3k{+}2}\left(s{+}\frac{3k^2{+}k}{2}\right)
\end{array}\right),
D_0{=}\left(\begin{array}{c}\sigma_{1,1}(s) \\\sigma_{1,2}(s)
\end{array}\right).
\end{equation}
\end{corollary}

Fuchs and Feigin \cite{FeFu} have not found general explicit formulas for singular vectors
of the type $S_{p,1}$ and $S_{p,2}$ but they managed however to find formulae for elements
$\sigma_{p,q}(j)$ for modules $F_{\lambda,\mu}$ (using other arguments). But  other graded $L_1$-modules are also can be important for applications \cite{Mill} and require explicit formulas for singular vectors $S_{p,1}$ and $S_{p,2}$. The main result of this article, together with the Benoit-Saint-Aubin theorem provides us with the formulae useful for cohomology calculations.


\begin{thebibliography}{A}

\bibitem{AstFu}
A.B. Astashkevich, D.B. Fuchs, {\it Asymptotic of singular vectors in Verma modules
over Virasoro Lie algebra}, Pacific Journal of Math., {\bf 177}, No 2 (1977), 201-209.

\bibitem{BDFIZ}
M.~Bauer, Ph.~Di Francesco, C.~Itzykson, J.-B.~Zuber, {\it
Covariant differential equations and singular vectors in Virasoro
representations}, Nuclear Physics, B362 (1991), 515-562.



\bibitem{BenSA}
L.~Benoit, Y.~Saint-Aubin, {\it Degenerate conformal field
theories and explicit expressions for some null vectors}, Phys.
Letters, {\bf 215(B)} (1988), 517--522.




\bibitem{G}
L.V.~Goncharova, \textit{Cohomology of Lie algebras of formal
vector fields on the line}, Funct. Anal. and Appl. \textbf{7}:2
(1973), 6--14; {\bf 7}:3 (1973), 33--44.

\bibitem{FeFu1}
B.L.~Feigin, D.B.~Fuchs, {\it Invariant skew-symmetric
differential operators on the line and Verma modules over the
Virasoro algebra}, Funct. Anal. and Appl., {\bf 16}:2 (1982),
114--126.

\bibitem{FeFu}
B.L.~Feigin, D.B.~Fuchs, {\it Verma modules over Virasoro
algebra}, Lecture Notes in Math. {\bf 1060} (1984), 230--245.

\bibitem{FeFu2}
B.L.~Feigin, D.B.~Fuchs, {\it Representations of the Virasoro
algebra}, in: "Representations of Lie groups and related topics",
Adv. Stud. Contemp. Math., {\bf 7}, Gordon \& Breach, 1991,
465--554.

\bibitem{FeFuRe}
B.L.~Feigin, D.B.~Fuchs, V.S.~Retakh, \textit{Massey operations in
the cohomology of the infinite dimensional Lie algebra $L_1$},
Lecture Notes in Math., {\bf 1346} (1988), 13--31.

\bibitem{Fu}
D.B. Fuchs, \textit{Cohomology of the infinite dimensional Lie
algebras}, Consultant Bureau, New York, 1987.


\bibitem{Fu2}
D.B.~Fuchs, \textit{Singular vectors over the Virasoro algebra and
extended Verma modules},Adv. Soviet Math.{\bf 17}, 65--74, Amer.
Math. Soc., Providence, RI, 1993.

\bibitem{Kac}
V.G.~Kac, {\it Contravariant form for infinite-dimensional Lie
algebras}, Lect. Notes in Phys., {\bf 94(B)} (1979), 441--445.

\bibitem{Mill}
D.V.~Millionshchikov, {\it Algebra of formal vector fields on the line and Buchstaber’s conjecture}, Funct. Anal. and Appl., (2009), {\bf 43}:4, 264--278.

\bibitem{RochWall}
A.~Rocha-Carridi, N.R.~Wallach, {\it Characters of irreducuble
representations of the algebra of vector fields on the circle},
Invent. Math., {\bf 72} (1983), 57--75.


\end{thebibliography}
\end{document}